\documentclass[12pt]{amsart}

\usepackage{amsmath,amsfonts,fullpage,cases,times}

\newtheorem{theorem}{Theorem}[section]
\newtheorem{proposition}[theorem]{Proposition}
\newtheorem{lemma}[theorem]{Lemma}

{
\theoremstyle{definition}
\newtheorem{definition}[theorem]{Definition}

}
{
\theoremstyle{remark}
\newtheorem{remark}[theorem]{Remark}
}
\numberwithin{equation}{section}

\renewcommand{\S}{\mathfrak{S}}

\newcommand{\ini}{\mathtt{ini}}

\newcommand{\ides}{\mathtt{ides}}
\newcommand{\imaj}{\mathtt{imaj}}

\newcommand{\exc}{\mathtt{exc}}
\newcommand{\maj}{\mathtt{maj}}
\newcommand{\fix}{\mathtt{fix}}
\newcommand{\des}{\mathtt{des}}
\newcommand{\aid}{\mathtt{aid}}
\newcommand{\pix}{\mathtt{pix}}
\newcommand{\inv}{\mathtt{inv}}
\newcommand{\lec}{\mathtt{lec}}
\newcommand{\aix}{\mathtt{aix}}
\newcommand{\mix}{\mathtt{mix}}
\newcommand{\das}{\mathtt{das}}
\newcommand{\eul}{\mathtt{eul}}
\newcommand{\mah}{\mathtt{mah}}

\newcommand{\Exc}{\mathtt{Exc}}

\newcommand{\Des}{\mathtt{Des}}
\newcommand{\Inv}{\mathtt{Inv}}
\newcommand{\Ai}{\mathtt{Ai}}
\newcommand{\ai}{\mathtt{ai}}

\author{Alexander Burstein}
\address{Department of Mathematics, Howard University, Washington, DC 20059 USA}
\email{aburstein@howard.edu}
\urladdr{http://www.alexanderburstein.org}

\date{February 14, 2014}

\title{On the distribution of some Euler-Mahonian statistics}


\begin{document}

\begin{abstract}
%\paragraph{Abstract.}
We give a direct combinatorial proof of the equidistribution of two pairs of permutation statistics, $(\des,\aid)$ and $(\lec,\inv)$, which have been previously shown to have the same joint distribution as $(\exc,\maj)$, the major index and the number of excedances of a permutation. Moreover, the triple $(\pix,\lec,\inv)$ was shown to have the same distribution as $(\fix,\exc,\maj)$, where $\fix$ is the number of fixed points of a permutation. We define a new statistic $\aix$ so that our bijection maps $(\pix,\lec,\inv)$ to $(\aix,\des,\aid)$. We also find an Eulerian partner $\das$ for a Mahonian statistic $\mix$ defined using mesh patterns, so that $(\das,\mix)$ is equidistributed with $(\des,\inv)$.

%\paragraph{R\'esum\'e.}

\end{abstract}

\maketitle

\section{Introduction} \label{sec:intro}

A \emph{combinatorial statistic} on a set $S$ is a map $\mathbf{f}:S\to\mathbb{N}^m$ for some integer $m\ge 0$. The \emph{distribution} of $\mathbf{f}$ is the map $\mathtt{d}_\mathbf{f}:\mathbb{N}^m\to\mathbb{N}$ with $\mathtt{d}_\mathbf{f}(\mathbf{i})=|\mathbf{f}^{-1}(\mathbf{i})|$ for $\mathbf{i}\in\mathbb{N}^m$, where $|\mathbf{f}^{-1}(\mathbf{i})|$ is the number of objects $s\in S$ such that $\mathbf{f}(s)=\mathbf{i}$. We say that statistics $\mathbf{f}$ and $\mathbf{g}$ are \emph{equidistributed} and write $\mathbf{f}\sim\mathbf{g}$ if $\mathtt{d}_\mathbf{f}=\mathtt{d}_\mathbf{g}$.

Let $\S_n$ be the set of permutations of $[n]=\{1,\dots,n\}$. The four classic combinatorial statistics on $\S_n$, the number of \emph{descents}, $\des$, the number of \emph{excedances}, $\exc$, the number of \emph{inversions}, $\inv$, and the \emph{major index}, $\maj$, are defined as follows:
\begin{alignat*}{2}
\Des\,\pi&=\{i:\ \pi(i)>\pi(i+1)\}, \quad &\des\,\pi&=|\Des\,\pi|,\\
\Exc\,\pi&=\{i:\ \pi(i)>i\}, \quad &\exc\,\pi&=|\Exc\,\pi|,\\
\Inv\,\pi&=\{(i,j):\ i<j \text{ and } \pi(i)>\pi(j)\}, \quad &\inv\,\pi&=|\Inv\,\pi|,\\
&\quad &\maj\,\pi&=\sum_{i\in\Des\,\pi}{i}.
\end{alignat*}

The sets $\Des\,\pi$ is called the \emph{descent set} of $\pi$, and its elements are called \emph{descents}. If $i$ is a descent of $\pi$, then $\pi(i)$ and $\pi(i+1)$ are called \emph{descent top} and \emph{descent bottom}, respectively. The terminology for the other two sets, $\Inv\,\pi$ and $\Exc\,\pi$, is similar. When the context is unambiguous, we may refer to the pair $\pi(i)\pi(i+1)$ as a descent or the pair $(\pi(i),\pi(j))$ as an inversion.

A statistic with the same distribution as $\des$ (such as $\exc$) is called \emph{Eulerian}, and a statistic with the same distribution as $\inv$ (such as $\maj$ \cite{MM1}) is called \emph{Mahonian}. If $\eul$ is Eulerian and $\mah$ is Mahonian, then the pair $(\eul,\mah)$ is called an Euler-Mahonian statistic.

A problem frequently considered since \cite{F1} is as follows: given a known Euler-Mahonian statistic $(\eul_1,\mah_1)$ and another Eulerian (resp. Mahonian) statistic $\eul_2$ (resp. $\mah_2$), to find its Mahonian (resp. Eulerian) partner $\mah_2$ (resp. $\eul_2$) so that $(\eul_1,\mah_1)\sim(\eul_2,\mah_2)$. In this paper, we will give two bijective proofs of equidistribution of two such pairs of bistatistics. In Section \ref{sec:des-aid-lec-inv}, we give a direct proof of a bijection between two statistics previously shown to have the same distribution as $(\exc,\maj)$, and in Section \ref{sec:des-inv-das-mix} we find an Eulerian partner $\das$ for a statistic $\mix$ recently defined by Br\"and\'en and Claesson \cite{BC} using mesh patterns so that $(\das,\mix)\sim(\des,\inv)$.

\section{$(\des,\aid)\sim(\lec,\inv)$} \label{sec:des-aid-lec-inv}

Of the four pairs $(\eul_1,\mah_1)$ involving $\des$ or $\exc$ and $\inv$ or $\maj$, the last to be considered was the pair $(\exc,\maj)$. First, Shareshian and Wachs \cite{SW} found a Mahonian statistic $\aid$ such that $(\exc,\maj)\sim(\des,\aid)$, and soon afterwards Foata and Han \cite{FH} proved that $(\exc,\maj)\sim(\lec,\inv)$ for an Eulerian statistic $\lec$ defined earlier by Gessel \cite{G} and related to the hook factorization of a permutation. In fact, Foata and Han proved a more refined result that $(\fix,\exc,\maj)\sim(\pix,\lec,\inv)$, where $\fix\,\pi$ is the number of fixed points of $\pi$ and $\pix\,\pi$ is another statistic related to hook factorization of $\pi$.

We will now define the $\aid$, $\lec$ and $\pix$. 

\begin{definition} \label{def:admissible}
An inversion $(i,j)\in\Inv\,\pi$ is \emph{admissible} if either $\pi(j)<\pi(j+1)$ or $\pi(j)>\pi(k)$ for some $i<k<j$. Let $\Ai\,\pi$ be the set of admissible inversions of $\pi$, and let
\[
\ai\,\pi=|\Ai\,\pi|, \qquad \aid\,\pi=\ai\,\pi+\des\,\pi.
\]
\end{definition}

\begin{definition} \label{def:hook}
A string $w=w_1w_2\dots w_r$ $(r\ge 2)$, over a totally ordered alphabet is a \emph{hook} if $w_1>w_2\le w_3\dots\le w_r$. Every string $\pi$ over $\mathbb{N}$ (and hence any permutation $\pi$) can be decomposed uniquely \cite{G} as $\pi=\pi_0\pi_1\dots\pi_k$ ($k\ge 0$), where $\pi_0$ is an nondecreasing string and each of $\pi_i$, $1\le i\le k$, is a hook. Then $\pi_0\pi_1\dots\pi_k$ is called the \emph{hook factorization} of $\pi$. 
\end{definition}

It is easy to see that the hook factorization is unique for any $\pi$, since either $\pi=\pi_0$ or we can recursively find the rightmost hook of $\pi$, which starts with the rightmost descent top of $\pi$. The statistics $\lec$ and $\pix$ are defined as follows:
\[
\lec\,\pi=\sum_{i=1}^{k}{\inv\,\pi_i}, \qquad \pix\,\pi=|\pi_0|,
\]
where $|\pi_0|$ is the length of $\pi_0$.

Shareshian and Wachs \cite{SW} gave a proof of $(\des,\aid)\sim(\exc,\maj)$ using tools from poset topology such as lexicographic shellability. Subsequently, Foata and Han \cite{FH} gave a two-step proof of $(\fix,\exc,\maj)\sim(\pix,\lec,\inv)$. The first step was a bijection on $\S_n$ showing that $(\fix,\exc,\maj)\sim(\pix,\lec,\imaj)$ (and, in fact, a more refined result that $(\fix,\exc,\des,\maj)\sim(\pix,\lec,\ides,\imaj)$), where $\imaj(\pi)=\maj(\pi^{-1})$ and $\ides(\pi)=\des(\pi^{-1})$, using Lyndon words and the word analogs of Kim-Zeng \cite{KZ} permutation decomposition and hook factorization. The second step was a bijection on $\S_n$ showing that $(\pix,\lec,\imaj)\sim(\pix,\lec,\inv)$.

Somewhat surprisingly, a direct bijective proof of $(\des,\aid)\sim(\lec,\inv)$ is simpler than any of the bijections mentioned above. We give such a proof and, in fact, find a new statistic $\aix$ that is a $\fix$-partner for $(\des,\aid)$, i.e. such that $(\aix,\des,\aid)\sim(\pix,\lec,\inv)\sim(\fix,\exc,\maj)$.

The statistic $\aix$ is defined as follows. Consider the set $\mathbb{N}^\ast$ of all strings in $\mathbb{N}$. Given a string $\pi\in\mathbb{N}^\ast$, let $m$ be the smallest letter in $\pi$ and let $\alpha$ be the maximal left prefix of $\pi$ not containing $m$, so that $\pi=\alpha m\beta$ for some string $\beta$. Then we recursively define $\aix\,\emptyset=0$ and, for $\pi\ne\emptyset$,
\begin{subnumcases}{\aix\,\pi=}
\aix\,\alpha,  & \text{ if } $\alpha\ne\emptyset,\ \beta\ne\emptyset$,\label{eq:aix-ab}\\
1+\aix\,\beta, & \text{ if } $\alpha=\emptyset$, \label{eq:aix-b}\\
0, & \text{ if } $\beta=\emptyset$. \label{eq:aix-a}
\end{subnumcases}
For example, $\aix(2589637\underline{1}4) = \aix(\underline{2}589637) = 1 + \aix(5896\underline{3}7) = 1 + \aix(\underline{5}896) = 1 + 1 + \aix(89\underline{6}) = 1 + 1 + 0 = 2$ (the smallest letters at each step are underlined).

\begin{proposition}
For any $\pi\in\mathbb{N}^\ast$, we have $\aix\,\pi\le 1+\pix\,\pi$.
\end{proposition}

\begin{proof}
The value of $\aix\,\pi$ is at most the length of $\rho$, the maximal nondecreasing left prefix of $\pi$. Since the leftmost hook of $\pi$ starts either at the leftmost descent or at the second leftmost descent (only if it immediately follows the leftmost descent), it follows that the length of $\rho$ is either $\pix\,\pi$ or $1+\pix\,\pi$.
\end{proof}

We also note that computations of statistics $\inv,\lec,\pix,\aid,\des,\aix$, involves only comparisons of values of letters or values of positions, but not values of a letter and a position (as in computation of $\exc$), so that these statistics can be extended to any string of distinct letters.

\subsection{The bijection}

Let $S$ be a set of distinct letters and $k\notin S$ be such that $S\cup\{k\}$ is totally ordered. Let $\tau$ be a permutation of $S$. Let $m$ be the smallest letter in $S\cup\{k\}$. Define a permutation $f(k,\tau)$ of $S\cup\{k\}$ recursively as follows: $f(k,\emptyset)=k$ and
\begin{subnumcases}{f(k,\tau)=}
f(k,\alpha)m\beta, &\qquad \text{if}\ $\tau = \alpha m\beta, k>m, \alpha\ne\emptyset, \beta\ne\emptyset$, \label{eq:f-ab}\\
f(k,\beta)m, &\qquad \text{if}\ $\tau = m\beta, k>m$, \label{eq:f-b}\\
km\alpha, &\qquad \text{if}\ $\tau = \alpha m, k>m$,  \label{eq:f-a}\\
k\tau, &\qquad \text{if} \ $k=m$.  \label{eq:f-m}
\end{subnumcases}
Now, for $\pi\in\S_n$, define $\phi_0(\pi)=\emptyset$ and $\phi_k(\pi)=f(\pi(n-k+1),\phi_{k-1}(\pi))$, $k=1,\dots,n$. Finally, let $\phi(\pi)=\phi_n(\pi)\in\S_n$. It is straightforward to see that $f$, and thus, $\phi$, are bijections.

Let $\ini\,\pi=\pi(1)$. Then we have that

\begin{theorem} \label{thm:aid-des-inv-lec}
$(\ini,\aix,\des,\aid)\,\phi(\pi)=(\ini,\pix,\lec,\inv)\,\pi$.
\end{theorem}

We will split the proof of the theorem into several parts.

\begin{lemma} \label{lem:ini}
$\ini\,\phi(\pi)=\ini\,\pi$.
\end{lemma}

\begin{proof}
Note that $f(k,\emptyset)=k$, so by the definition of $f$ and induction on the size of $\tau$ we get that $f(k,\tau)$ starts with $k$. Thus, $\phi(\pi)$ starts with $\pi(1)$.
\end{proof}

Given a string $\pi$ over a totally ordered alphabet define $k$-suffix of $\pi$, $s_k(\pi)$, to be the block of $k$ rightmost letters of $\pi$. Also, define $\pi_{<k}$ (resp. $\pi_{>k}$) to be the subsequence of $\pi$ consisting of letters of $\pi$ that are less (resp. greater) than $k$.

\begin{lemma} \label{lem:aid}
$\aid\,f(k,\tau)=\aid\,\tau+|\tau_{<k}|$.
\end{lemma}

\begin{proof}
We will prove this lemma by induction on the length of $\tau$. Clearly, the lemma is true for $\tau=\emptyset$. Assume that the lemma holds for all strings of distinct letters of length less than $|\tau|$. Let $m=\min\,\tau$ and consider each case in the definition of $f(k,\tau)$.

\emph{Case (a).} Suppose that $\tau=\alpha m\beta$, $k>m$, $\alpha\ne\emptyset$, $\beta\ne\emptyset$. Then $f(k,\tau)=f(k,\alpha)m\beta$, so by Lemma \ref{lem:ini}, $f(k,\alpha m\beta)=k\hat\alpha m\beta$ for some permutation $\hat\alpha$ of $\alpha$. By induction (since $|\alpha|<|\tau|$), we have
\[
\aid\,f(k,\alpha)=\aid\,\alpha+|\alpha_{<k}|.
\]
Consider the inversions $ab$ in $\tau$ that are \emph{from $\alpha$ to $m\beta$}, i.e. those where the inversion top is $a\in\alpha$ and the inversion bottom is $b\in m\beta$ (so $a>b$). If $b=m$, then it is followed by an ascent, and hence any inversion with inversion bottom $m$ is admissible (and the number of such (admissible) inversions in $\tau$ is $|\alpha|$). If $b\in\beta$, then $m<b$ and $m$ is between $a$ and $b$ in $\tau$, so the inversion $ab$ is admissible. Thus, all inversions from $\alpha$ to $m\beta$ are admissible.

Since $\hat\alpha$ is a permutation of $\alpha$, we likewise have that all inversions in $f(k,\tau)$ from $\hat\alpha$ to $m\beta$ are admissible, and in fact, are the same inversions as the inversions from $\alpha$ to $m\beta$ in $\tau$. Moreover, since $\alpha>m$ (i.e. every letter in $\alpha$ is greater than $m$) and $f$ does not change the suffix $m\beta$ of $\tau$, it follows that the number of admissible inversions in $m\beta$ and the number of descents with descent bottoms in $m\beta$ are the same in $\tau$ and $f(k,\tau)$.

Thus, the only remaining pairs left to consider are inversions from $k$ to $m\beta$. As above, we see that all inversions from $k$ to $m\beta$ are admissible, and the number of such inversions is exactly $|m\beta_{<k}|$. Therefore,
\[
\aid\,f(k,\tau)-\aid\,\tau=|\alpha_{<k}|+|m\beta_{<k}|=|\alpha_{<k}m\beta_{<k}|=|\tau_{<k}|,
\]
as desired.

\emph{Case (b).} Suppose that $\tau=m\beta$ and $k>m$. Then $f(k,\tau)=f(k,\beta)m=k\hat\beta m$ for some permutation $\hat\beta$ of $\beta$. As before, we have by induction that
\[
\aid\,f(k,\beta)=\aid\,\beta+|\beta_{<k}|.
\]
Since $k\hat\beta>m$ and $m$ is last in $k\hat\beta m$, it follows that no admissible inversion ends on $m$. Thus, $\ai\,f(k,\beta)m=\ai\,f(k,\beta)$ and $\des\,f(k,\beta)m=\des\,f(k,\beta)+1$, where 1 counts the last descent to $m$. Finally, $\aid\,\tau=\aid\,m\beta=\aid\,\beta$ since $m<\beta$ and hence no inversion (or descent) of $\tau$ begins with $m$. Therefore,
\[
\aid\,f(k,\tau)=\aid\,f(k,\beta)+1=\aid\,\beta+|\beta_{<k}|+1=\aid\,m\beta+|m\beta_{<k}|=\aid\,\tau+|\tau_{<k}|.
\]

\emph{Case (c).} Suppose that $\tau =\alpha m$, $k>m$. Then $f(k,\tau)=km\alpha$. Thus, the descents of $f(k,\tau)$ are obtained from descents of $\tau$ by replacing the descent from the right letter of $\alpha$ to $m$ with the descent $km$, so $\des\,f(k,\tau)=\des\,\tau$. As Case (b), no admissible inversion of $\tau$ ends on $m$, and, as in Cases (a) and (b), all inversions from $k$ to $m\alpha$ are admissible. Thus,
\[
\ai\,f(k,\tau)=\ai\,km\alpha=\ai\,m\alpha +|m\alpha_{<k}|=\ai\,\alpha m +|\alpha_{<k}m|=\ai\,\tau+|\tau_{<k}|,
\]
so
\[
\aid\,f(k,\tau)=\ai\,f(k,\tau)+\des\,f(k,\tau)=\ai\,\tau+|\tau_{<k}|+\des\,\tau=\aid\,\tau+|\tau_{<k}|.
\]

\emph{Case (d).} If $k<\tau$, then no inversion (or descent) of $f(k,\tau)=k\tau$ starts with $k$, and $|\tau_{<k}|=0$, so $\aid\,f(k,\tau)=\aid\,\tau=\aid\,\tau+\tau_{<k}$. This ends the proof.
\end{proof}

\begin{lemma} \label{lem:aid-inv}
$\aid\,\phi(\pi)=\inv\,\pi$.
\end{lemma}

\begin{proof}
Applying Lemma \ref{lem:aid} repeatedly, we obtain
\[
\aid\,\phi(\pi)=\sum_{k=0}^{n-1}{|\phi_k(\pi)_{<\pi(n-k)}|}.
\]
But each $\phi_k(\pi)$ is a permutation of $s_k(\pi)$, so
\[
\aid\,\phi(\pi)=\sum_{k=0}^{n-1}{|s_k(\pi)_{<\pi(n-k)}|}.
\]
Each summand on the right is the number of inversions of $\pi$ with inversion top $\pi(n-k)$. Summing over $k=0,1\dots,n-1$, we get $\aid\,\phi(\pi)=\inv(\pi)$, as desired.
\end{proof}

Consider the descents of $\tau$ and $f(k,\tau)$ in each case of the definition of $f$. In case \eqref{eq:f-ab}, we have $\alpha>m$ and $f(k,\tau)=f(k,\alpha m\beta)=f(k,\alpha)m\beta$, so the descent bottoms in the right prefix $m\beta$ of both $\tau$ and $f(k,\tau)$ are the same, and hence
\[
\des\,f(k,\tau)-\des\,\tau=\des\,f(k,\alpha)-\des\,\alpha.
\]
Note that in this case $\aix\,\tau=\aix\,\alpha$ and $\aix\,f(k,\tau)=\aix\,f(k,\alpha)$.

In case \eqref{eq:f-b}, $\des\,\tau=\des\,m\beta=\des\,\beta$ since $m<\beta$. However, $\des\,f(k,\tau)=\des\,f(k,\beta)m=\des\,f(k,\beta)+1$ since $f(k,\beta)=k\hat\beta$ for some permutation $\hat\beta$ of $\beta$. Thus,
\[
\des\,f(k,\tau)-\des\,\tau=\des\,f(k,\beta)-\des\,\beta +1.
\]
Note that in this case $\aix\,f(k,\tau)=0$, and $\aix\,\tau=1+\aix\,\beta>0$.

In case \eqref{eq:f-a}, let $a$ be the last letter of $\alpha$. Then the descents of $f(k,\tau)=km\alpha$, $\alpha\ne\emptyset$ are obtained from the descents of $\tau=\alpha m$ by replacing the descent $am$ with the descent $km$. Thus, $\des\,f(k,\tau)=\des\,\tau=\des\,\alpha+1$, and hence
\[
\des\,f(k,\tau)-\des\,\tau=0.
\]
Note that in this case $\aix\,\tau=0$ and $\aix\,f(k,\tau)=\aix\,k=1=1+\aix\,\tau$.

In case \eqref{eq:f-m}, $f(k,\tau)=k\tau$, and $k<\tau$, so $\des\,f(k,\tau)=\des\,\tau$, and hence again
\[
\des\,f(k,\tau)-\des\,\tau=0.
\]
Note that in this case $\aix\,f(k,\tau)=\aix\,\tau+1>0$.

Finally, $\des\,f(k,\emptyset)-\des\,\emptyset=1-0=1$. Thus, we can see by induction on the length of $\tau$ that
\[
\des\,f(k,\tau)-\des\,\tau\ge 0
\]
for any string $\tau$ of distinct letters, and the difference stays the same or increases by 1 with each application of rules \eqref{eq:f-ab} or \eqref{eq:f-b}, respectively.

\begin{lemma} \label{lem:same-des}
We have $\des\,f(k,\tau)=\des\,\tau$ if and only if $\aix\,f(k,\tau)=\aix\,\tau+1>0$, and $\des\,f(k,\tau)>\des\,\tau$ if and only if $\aix\,f(k,\tau)=0$.
\end{lemma}

\begin{proof}
\emph{Case 1.} Suppose that $\des\,f(k,\tau)=\des\,\tau$. Then it follows from the above argument that the computation of $f(k,\tau)$ involves no application of \eqref{eq:f-b}, i.e. a repeated application of \eqref{eq:f-ab} (possibly zero times) followed by a single application of \eqref{eq:f-a} or \eqref{eq:f-m} or $f(k,\emptyset)=k$. The conditions in the case \eqref{eq:f-ab} are the same as in the case \eqref{eq:aix-ab}, so applying \eqref{eq:f-ab} repeatedly, we obtain either
\begin{itemize}
\item a prefix $\alpha'm'$ of $\tau$ such that $\alpha\ne\emptyset$, $\alpha'>m'$, $k>m'$, $\aix\,\tau=\aix\,\alpha'm'$ and $\aix\,f(k,\tau)=\aix\,f(k,\alpha'm')$, or
\item a prefix $\alpha''$ of $\tau$ such that $k<\alpha''$, $\aix\,\tau=\aix\,\alpha''$ and $\aix\,f(k,\tau)=\aix\,f(k,\alpha'')$.
\end{itemize}
In the former case, we have $\aix\,\tau=\aix\,\alpha'm'=0$ and $\aix\,f(k,\tau)=\aix\,f(k,\alpha'm')=\aix\,km'\alpha'=\aix\,k=1=1+\aix\,\tau$. In the latter case, we have $\aix\,f(k,\alpha'')=\aix,k\alpha''=1+\aix\,\alpha''=1+\aix\,\tau$. Thus, in either case, $\des\,f(k,\tau)=\des\,\tau$ implies $\aix\,f(k,\tau)=\aix\,\tau+1$. The converse is proved similarly.

\smallskip

\emph{Case 2.} Suppose that $\des\,f(k,\tau)>\des\,\tau$. Then the computation of $f(k,\tau)$ starts with a repeated application of \eqref{eq:f-ab} (possibly zero times) followed by an application of \eqref{eq:f-b} (after which the process may still continue). Thus, as before, after repeated application of \eqref{eq:f-ab}, we obtain a prefix $m'\beta'$ of $\tau$ such that $k>m'$, $m'<\beta'$ and $\aix\,f(k,\tau)=\aix\,f(k,m'\beta')=\aix\,f(k,\beta')m'$. But $f(k,\beta')=k\hat\beta'$ for some permutation $\hat\beta'$ of $\beta'$, so $f(k,\beta')>m'$, and hence $\aix\,f(k,\beta')m'=0$, which in turn implies that $\aix\,f(k,\tau)=0$, as desired. The converse is proved similarly.
\end{proof}

\begin{lemma} \label{lem:no-2-0s-in-a-row}
If $\aix\,\tau=0$, then for any $k$, we have $\aix\,f(k,\tau)=1$ and $\des\,f(k,\tau)=\des\,\tau$.
\end{lemma}

\begin{proof}
The lemma is obviously true for $\tau=\emptyset$. Suppose $\tau\ne\emptyset$. Since $\aix\,\tau=0$, it follows that $\tau=\alpha m_0 m_1\beta_1\dots m_r\beta_r$, where $\alpha\ne\emptyset$, $\alpha>m_0$, $\beta_i\ne\emptyset$ and $m_i<\beta_i$ for all $i=1,\dots,r$, and $m_0>m_1>\dots>m_n$. If $k>m_0$, then applying \eqref{eq:f-ab} repeatedly followed by \eqref{eq:f-a}, we obtain
\[
f(k,\tau)=f(k,\alpha m_0 m_1\beta_1\dots m_r\beta_r)=f(k,\alpha m_0)m_1\beta_1\dots m_r\beta_r=km_0\alpha m_1\beta_1\dots m_r\beta_r
\]
so that $\aix\,f(k,\tau)=\aix\,km_0\alpha=\aix\,k=1$. Also, all descent bottoms of $\tau$ and $f(k,\tau)$ are the same (including $m_0$), so $\des\,f(k,\tau)=\des\,\tau$.

Suppose that $k<m_0$, and let $j$ be the maximal such that $k<m_j$. Then $k<\alpha m_0m_1\beta_1\dots m_j\beta_j$, so
\[
\begin{split}
f(k,\tau)&=f(k,\alpha m_0 m_1\beta_1\dots m_r\beta_r)\\
&=f(k,\alpha m_0 m_1\beta_1\dots m_j\beta_j)m_{j+1}\beta_{j+1}\dots m_r\beta_r\\
&=k\alpha m_0 m_1\beta_1\dots m_j\beta_j m_{j+1}\beta_{j+1}\dots m_r\beta_r\\
&=k\tau.
\end{split}
\]
Therefore, $f(k,\tau)=k\tau$ starts with an ascent, so $\des\,f(k,\tau)=\des\,k\tau=\des\,\tau$ and hence $\aix\,f(k,\tau)=1+\aix\,\tau=1$ by Lemma \ref{lem:same-des}.
\end{proof}

\begin{lemma} \label{lem:delete-second}
Suppose that $\aix\,f(k,\tau)=0$ and $\tau=f(l,\sigma)$ for some letter $l$ and string $\sigma$. Then $\des\,f(k,\tau)=1+\des\,f(k,\sigma)$.
\end{lemma}

\begin{proof}
Note that $\aix\,\tau\ge 1$ since otherwise $\aix\,f(k,\tau)=1$. In particular, $\tau\ne\emptyset$, so indeed there is a letter $l$ and a string $\sigma$ such that $\tau=f(l,\sigma)$.

Since $\tau=f(l,\sigma)$, it follows that $\tau$ starts with $l$. Let $l=m_0>m_1>\dots>m_r$ be the left-to-right minima of $\tau$. Then $\tau=m_0\tau_0 m_1\tau_1\dots m_r\tau_r$ with $\tau_i>m_i$ for all $i=0,1\dots,r$. We also have that $\tau_i\ne\emptyset$ for $i\ge 1$ since otherwise $\aix\,\tau=0$. Therefore,
\[
f(k,\tau)=f(k,m_0\tau_0)m_1\tau_1\dots m_r\tau_r=f(k,l\tau_0)m_1\tau_1\dots m_r\tau_r,
\]
so $\aix\,f(k,\tau)=\aix,f(k,l\tau_0)$. If $k<l$, then $k<l\tau_0$, so $f(k,l\tau_0)=kl\tau_0$ and
\[
\aix\,f(k,l\tau_0)=1+\aix\,l\tau_0=1+\aix\,\tau>0,
\]
which contradicts our assumption. Therefore, $k>l$.

Since $\aix\,f(l,\sigma)=\aix\,\tau>0$, it follows that the recursive computation of $f(l,\sigma)$ involves no application of \eqref{eq:f-b}. Thus, we have two cases:
\begin{itemize}
\item $\sigma=\alpha l_1\beta_1\dots l_s\beta_s$, where $l>l_1>\dots>l_s$, $\alpha\ne\emptyset$, $\alpha>l$, $\beta_i\ne\emptyset$ and $\beta_i>l_i$ for $i=1,\dots,s$.

\item $\sigma=\alpha l_0l_1\beta_1\dots l_s\beta_s$, where $l>l_0>l_1>\dots>l_s$, $\alpha\ne\emptyset$, $\alpha>l$, $\beta_i\ne\emptyset$ and $\beta_i>l_i$ for $i=1,\dots,s$.

\end{itemize}

Let $\beta=l_1\beta_1\dots l_s\beta_s$. In the first case, we have
\[
\begin{split}
\tau=f(l,\sigma)&=f(l,\alpha\beta)=f(l,\alpha)\beta=l\alpha\beta=l\sigma\\
f(k,\tau)&=f(k,l\alpha\beta)=f(k,l\alpha)\beta=f(k,\alpha)l\beta\\
f(k,\sigma)&=f(k,\alpha\beta)=f(k,\alpha)\beta.
\end{split}
\]
Note that $\ini\,\beta=l_1<l$. Also note that $f(k,\alpha)l\beta=k\hat\alpha l\beta$ for some permutation $\hat\alpha$ of $\alpha$. Since $\alpha>l$, it follows that $\hat\alpha>l$. Let $a$ be the last letter of $f(k,\alpha)$. Then the descents of $f(k,\alpha)l\beta$ are obtained from the descents of $f(k,\alpha)\beta$ by replacing the descent $al_1$ with the descents $al$ and $ll_1$. Therefore, we have $\des\,f(k,\tau)=\des\,f(k,\sigma)+1$ as desired.

In the second case, we have
\[
\begin{split}
\tau=f(l,\sigma)&=f(l,\alpha l_0\beta)=f(l,\alpha l_0)\beta=ll_0\alpha\beta\\
f(k,\tau)&=f(k,ll_0\alpha\beta)=f(k,ll_0\alpha)\beta=f(k,l)l_0\alpha\beta=kll_0\alpha\beta\\
f(k,\sigma)&=f(k,\alpha l_0\beta)=f(k,\alpha l_0)\beta=kl_0\alpha\beta.
\end{split}
\]
Since $k>l>l_0$, it is easy to see that $\des\,f(k,\tau)=\des\,f(k,\sigma)+1$. This ends the proof.
\end{proof}

\begin{lemma} \label{lem:des-lec}
$(\aix,\des)\,\phi(\pi)=(\pix,\lec)\,\pi$.
\end{lemma}

\begin{proof}
The proof is by induction on the length of $\pi$. The result is obviously true for $\pi=\emptyset$. Define $g(k,\tau)=k\tau$ for a string $\tau$ of distinct elements and an element $k$ not in the alphabet of $\tau$. Then it is easy to see that the results of Lemmas \ref{lem:same-des}, \ref{lem:no-2-0s-in-a-row} and \ref{lem:delete-second} hold if we replace $f$ with $g$, $\aix$ with $\pix$, and $\des$ with $\lec$. This implies the lemma and thus finishes the proof of Theorem \ref{thm:aid-des-inv-lec}.
\end{proof}

\begin{remark}
We note that a statistic $\mathtt{rix}$ similar to $\aix$ (up to an easy transformation) has been independently defined by Z. Lin \cite{L1}.
\end{remark}

It would be interesting to construct a direct bijection on permutations that maps $(\aix,\des,\aid)$ to $(\fix,\exc,\maj)$.

\begin{remark}
\emph{Rawlings major index} $r\maj$ is a Mahonian statistics that interpolates between $\maj$ and $\inv$, and is defined as follows:
\[
\begin{split}
\Des_r(\pi)&=\{i\in\Des(\pi)\ :\ \pi(i)-\pi(i+1)\ge r\}\\
\Inv_r(\pi)&=\{(i,j)\in\Inv(\pi)\ :\ \pi(i)-\pi(j)<r\}\\
r\maj(\pi)&=\sum_{i\in\Des_r(\pi)}{i}+|\Inv_r(\pi)|
\end{split}
\]
Note that on $\S_n$, $1\maj=\maj$, $n\maj=\inv$, and $|\Inv_2(\pi)|=\ides(\pi)=\des(\pi^{-1})$. It is known \cite{W} that $(\ides,2\maj)\sim(\exc,\maj)$. It would be interesting to find a $\fix$-partner $2\fix$ for $(\ides,2\maj)$ so that $(2\fix,\ides,2\maj)\sim(\fix,\exc,\maj)$. Continuing in the same vein, for $3\le r\le n-1$, it would be interesting to find the interpolating statistics $r\fix$ and $r\exc$ so that $(\fix,\exc,\maj)\sim(r\fix,r\exc,r\maj)\sim(\pix,\lec,\inv)$.
\end{remark}

\section{$(\das,\mix)\sim(\des,\inv)$} \label{sec:des-inv-das-mix}

A Mahonian statistic $\mix$ counting some inversions and some noninversions has been defined by P. Br\"and\'en, A. Claesson \cite{BC}. Even though it was originally defined using \emph{mesh patterns}, it may be easily defined without using those. The statistic $\mix$ counts pairs defined on a permutation $\pi$ as follows:
\begin{itemize}

\item inversions $(\pi(i),\pi(j))$ such that $\pi(i)$ is a left-to-right maximum of $\pi$, and

\item non-inversions $(\pi(i),\pi(j))$ such that there is a (left-to-right-maximum) $\pi(k)$ with $k<i$ and $\pi(k)>\pi(j)$.

\end{itemize}

We note that, in fact, our definition of $\mix$ is the reversal of the $\mix$ as originally defined in \cite{BC}. However, we think our definition is preferable, since we have $\mix(id_n)=0$, rather than $\mix(id_n)=n-1$ under the original definition.

There is also a direct bijection given in \cite{BC} that takes $\inv$ to $\mix$. Making the necessary minor changes to account for the difference in definitions mentioned above, we describe it as follows.

Let $M,I\in[n]$ be such that $|M|=|I|$ and $n\in M$, $1\in I$, and let $\S_n(M,I)$ be the set of permutations in $\S_n$ that have left-to-right maxima exactly at the positions indexed by $I$, and set of values of the left-to-right maxima equal to $M$.

Let $M=\{m_1<\dots<m_k\}$, and let $B_i$ be the set of entries of $\pi$ that are smaller than and to the right of $m_i$. Also, for $S\subseteq[n]$, let  $\psi_S(\pi)$ be the result of reversing the subword of $\pi$ that is a permutation on $S$. Then define
\[
\psi=\psi_{B_1}\circ\psi_{B_2\cap B_1}\circ\dots\circ\psi_{B_{k-1}}\circ\psi_{B_k\cap B_{k-1}}\circ \psi_{B_k}.
\]
Then we have \cite{BC} that $\psi$ is an involution and $\mix\,\psi(\pi)=\inv\,\pi$ (and vice versa).

We observe that there is a natural Eulerian partner $\das$ (a mix of descents and ascents) for $\mix$ such that is a $(\das,\mix)\sim(\des,\inv)$. Let $\das\,\pi$ be the number of positions $i\in[n-1]$ of $\pi$ such that
\begin{itemize}

\item $\pi(i)\pi(i+1)$ is a descent, and $\pi(i)$ is a left-to-right maximum of $\pi$, or

\item $\pi(i)\pi(i+1)$ is an ascent, and there is a (left-to-right-maximum) $\pi(k)$ with $k<i$ and $\pi(k)>\pi(i+1)$.

\end{itemize}

\begin{theorem}
$(\das,\mix)\,\psi(\pi)=(\des,\inv)\,\pi$.
\end{theorem}

\begin{proof}
%We omit the proof in this extended abstract. 
The proof is easily constructed by induction on $k$, following along the lines of the proof of Theorem 10 in \cite{BC}. In fact, our extension of that proof is so routine that we leave it as an exercise for the reader.
\end{proof}

\begin{remark}
We also note that a restriction of the map $\psi$ yields Krattenthaler's bijection \cite{K1} between 321-avoiding and 312-avoiding permutations on $\S_n$ using Dyck paths (modified up to the suitable reversal and complementation symmetries). 
%Again, we omit (a rather straightforward) proof from this extended abstract.
\end{remark}

\end{document}